\documentclass[12pt]{amsart}
\usepackage{amsmath,amsthm,amsfonts,amssymb,mathrsfs}
\date{\today}

\usepackage{color}
\input xymatrix
\xyoption{all}

\usepackage{hyperref}

\newtheorem{theorem}{Theorem}[section]

\newtheorem{proposition}[theorem]{Proposition}
\newtheorem{corollary}[theorem]{Corollary}
\newtheorem{lemma}[theorem]{Lemma}
\theoremstyle{definition}
\newtheorem{example}[theorem]{Example}
\newtheorem{remark}[theorem]{Remark}

\begin{document}

\title[On a semitopological polycyclic monoid]{On a semitopological polycyclic monoid}

\author{Serhii Bardyla and Oleg Gutik}
\address{Faculty of Mathematics, National University of Lviv,
Universytetska 1, Lviv, 79000, Ukraine}
\email{sbardyla@yahoo.com, o\underline{\hskip5pt}\,gutik@franko.lviv.ua,
ovgutik@yahoo.com}



\begin{abstract}
We study algebraic structure of the $\lambda$-polycyclic monoid $P_{\lambda}$ and its topologizations. We show that the $\lambda$-polycyclic monoid for an infinite cardinal $\lambda\geqslant 2$ has similar algebraic properties so has the polycyclic monoid $P_n$ with finitely many $n\geqslant 2$ generators. In particular we prove that for every infinite cardinal $\lambda$ the polycyclic monoid $P_{\lambda}$ is a congruence-free combinatorial $0$-bisimple $0$-$E$-unitary inverse semigroup. Also we show that every non-zero element $x$ is an isolated point in $(P_{\lambda},\tau)$ for every Hausdorff topology $\tau$ on $P_{\lambda}$, such that $(P_{\lambda},\tau)$ is a semitopological semigroup, and every locally compact Hausdorff semigroup topology on $P_\lambda$ is discrete. The last statement extends results of the paper \cite{Mesyan-Mitchell-Morayne-Peresse-20??} obtaining for topological inverse graph semigroups. We describe all feebly compact topologies $\tau$ on $P_{\lambda}$ such that $\left(P_{\lambda},\tau\right)$ is a semitopological semigroup and its Bohr compactification as a topological semigroup. We prove that for every cardinal $\lambda\geqslant 2$ any continuous homomorphism from a topological semigroup $P_\lambda$ into an arbitrary countably compact topological semigroup is annihilating and there exists no a Hausdorff feebly compact topological semigroup which contains $P_{\lambda}$ as a dense subsemigroup.
\end{abstract}

\maketitle

\section{Introduction and preliminaries}

In this paper all topological spaces will be assumed to be
Hausdorff. We shall follow the terminology of~\cite{Carruth-Hildebrant-Koch-1983-1986, Clifford-Preston-1961-1967,
Engelking-1989, Lawson-1998}. If $A$ is a subset of a topological
space $X$, then we denote the closure of the set $A$ in $X$ by
$\operatorname{cl}_X(A)$. By $\omega$ we denote the first infinite
cardinal.

A semigroup $S$ is called an \emph{inverse semigroup} if every $a$
in $S$ possesses an unique inverse, i.e. if there exists an unique
element $a^{-1}$ in $S$ such that
\begin{equation*}
    aa^{-1}a=a \qquad \mbox{and} \qquad a^{-1}aa^{-1}=a^{-1}.
\end{equation*}
A map which associates to any element of an inverse semigroup its
inverse is called the \emph{inversion}.

A \emph{band} is a semigroup of idempotents. If $S$ is a semigroup, then we shall denote the subset of all idempotents in $S$ by $E(S)$. If $S$ is an inverse semigroup, then $E(S)$ is closed under multiplication. The semigroup operation on $S$
determines the following partial order $\leqslant$ on $E(S)$: $e\leqslant f$ if and only if $ef=fe=e$. This order is called the {\em natural partial order} on $E(S)$. A \emph{semilattice} is a commutative semigroup of idempotents. A semilattice $E$ is called {\em linearly ordered} or a \emph{chain} if its natural order is a linear order. A \emph{maximal chain} of a semilattice $E$ is a chain which is properly contained in no other chain of $E$. The Axiom of Choice implies the existence of maximal chains in any partially ordered set. According to \cite[Definition~II.5.12]{Petrich-1984} chain $L$ is called $\omega$-chain if $L$ is isomorphic to $\{0,-1,-2,-3,\ldots\}$ with the usual order $\leqslant$. Let $E$ be a semilattice and $e\in E$. We denote ${\downarrow} e=\{ f\in E\mid f\leqslant e\}$
and ${\uparrow} e=\{ f\in E\mid e\leqslant f\}$.

If $S$ is a semigroup, then we shall denote by $\mathscr{R}$,
$\mathscr{L}$, $\mathscr{J}$, $\mathscr{D}$ and $\mathscr{H}$ the
Green relations on $S$ (see \cite{GreenJ1951} or \cite[Section~2.1]{Clifford-Preston-1961-1967}):
\begin{center}
\begin{tabular}{rcl}
  $a\mathscr{R}b$ & if and only if & $aS^1=bS^1$; \\
  $a\mathscr{L}b$ & if and only if & $S^1a=S^1b$; \\
  $a\mathscr{J}b$ & if and only if & $S^1aS^1=S^1bS^1$; \\
    & $\mathscr{D}=\mathscr{L}{\circ}\mathscr{R}=\mathscr{R}{\circ}\mathscr{L}$; &\\
    & $\mathscr{H}=\mathscr{L}\cap\mathscr{R}$. &\\
\end{tabular}
\end{center}

A semigroup $S$ is said to be:
\begin{itemize}
  \item[$\bullet$] \emph{simple} if $S$ has no proper two-sided ideals, which is  equivalent to $\mathscr{J}=S\times S$ in $S$;
  \item[$\bullet$] \emph{$0$-simple} if $S$ has a zero and $S$ contains no proper two-sided ideals distinct from the zero;
  \item[$\bullet$] \emph{bisimple} if $S$ contains a unique $\mathscr{D}$-class, i.e., $\mathscr{D}=S\times S$ in $S$;
  \item[$\bullet$] \emph{$0$-bisimple} if $S$ has a zero and $S$ contains two $\mathscr{D}$-classes: $\{0\}$ and $S\setminus\{0\}$;
  \item[$\bullet$] \emph{congruence-free} if $S$ has only identity and universal congruences.
\end{itemize}

An inverse semigroup $S$ is said to be
\begin{itemize}
  \item[$\bullet$] \emph{combinatorial} if $\mathscr{H}$ is the equality relation on $S$;
  \item[$\bullet$] \emph{$E$-unitary} if for any idempotents $e,f\in S$ the equality $ex=f$ implies that $x\in E(S)$;
  \item[$\bullet$] \emph{$0$-$E$-unitary} if $S$ has a zero and for any non-zero idempotents $e,f\in S$ the equality $ex=f$ implies that $x\in E(S)$.
\end{itemize}

The bicyclic monoid ${\mathscr{C}}(p,q)$ is the semigroup with the identity $1$ generated by two elements $p$ and $q$ subjected only to the condition $pq=1$. The distinct elements of ${\mathscr{C}}(p,q)$ are exhibited in the following useful array
\begin{equation*}
\begin{array}{ccccc}
  1      & p      & p^2    & p^3    & \cdots \\
  q      & qp     & qp^2   & qp^3   & \cdots \\
  q^2    & q^2p   & q^2p^2 & q^2p^3 & \cdots \\
  q^3    & q^3p   & q^3p^2 & q^3p^3 & \cdots \\
  \vdots & \vdots & \vdots & \vdots & \ddots \\
\end{array}
\end{equation*}
and the semigroup operation on ${\mathscr{C}}(p,q)$ is determined as
follows:
\begin{equation*}
    q^kp^l\cdot q^mp^n=q^{k+m-\min\{l,m\}}p^{l+n-\min\{l,m\}}.
\end{equation*}
It is well known that the bicyclic monoid ${\mathscr{C}}(p,q)$ is a bisimple (and hence simple) combinatorial $E$-unitary inverse semigroup and every non-trivial congruence on ${\mathscr{C}}(p,q)$ is a group congruence \cite{Clifford-Preston-1961-1967}. Also the nice Andersen Theorem states that \emph{a simple semigroup $S$ with an idempotent is completely simple if and only if $S$ does not contains an isomorphic copy of the bicyclic semigroup} (see \cite{Andersen-1952} and \cite[Theorem~2.54]{Clifford-Preston-1961-1967}).

Let $\lambda$ be a non-zero cardinal. On the set
 $
 B_{\lambda}=(\lambda\times\lambda)\cup\{ 0\}
 $,
where $0\notin\lambda\times\lambda$, we define the semigroup
operation ``$\, \cdot\, $'' as follows
\begin{equation*}
(a, b)\cdot(c, d)=
\left\{
  \begin{array}{cl}
    (a, d), & \hbox{ if~ } b=c;\\
    0, & \hbox{ if~ } b\neq c,
  \end{array}
\right.
\end{equation*}
and $(a, b)\cdot 0=0\cdot(a, b)=0\cdot 0=0$ for $a,b,c,d\in\lambda$. The semigroup $B_{\lambda}$ is called the \emph{semigroup of $\lambda{\times}\lambda$-matrix units}~(see \cite{Clifford-Preston-1961-1967}).

In 1970 Nivat and Perrot proposed the following generalization of the bicyclic monoid (see \cite{Nivat-Perrot-1970} and \cite[Section~9.3]{Lawson-1998}). For a non-zero cardinal $\lambda$, the polycyclic monoid $P_\lambda$ on $\lambda$ generators is the semigroup with zero given by the presentation:
\begin{equation*}
    P_\lambda=\left\langle \left\{p_i\right\}_{i\in\lambda}, \left\{p_i^{-1}\right\}_{i\in\lambda}\mid p_i p_i^{-1}=1, p_ip_j^{-1}=0 \hbox{~for~} i\neq j\right\rangle.
\end{equation*}
It is obvious that in the case when $\lambda=1$ the semigroup $P_1$ is isomorphic to the bicyclic semigroup with adjoined zero. For every finite non-zero cardinal $\lambda=n$ the polycyclic monoid $P_n$ is a congruence free, combinatorial, $0$-bisimple, $0$-$E$-unitary inverse semigroup (see \cite[Section~9.3]{Lawson-1998}).

We recall that a topological space $X$ is said to be:
\begin{itemize}
  \item \emph{compact} if each open cover of $X$ has a finite subcover;
  \item \emph{countably compact} if each open countable cover of $X$ has a finite subcover;
  \item \emph{countably compact at a subset} $A\subseteq X$ if every infinite subset $B\subseteq A$  has  an  accumulation  point $x$ in $X$;
  \item \emph{countably pracompact} if there exists a dense subset $A$ in $X$  such that $X$ is countably compact at $A$;
  \item \emph{feebly compact} if each locally finite open cover of $X$ is finite.
\end{itemize}
According to Theorem~3.10.22 of \cite{Engelking-1989}, a Tychonoff topological space $X$ is feebly compact if and only if each continuous real-valued function on $X$ is bounded, i.e., $X$ is pseudocompact. Also, a Hausdorff topological space $X$ is feebly compact if and only if every locally finite family of non-empty open subsets of $X$ is finite. Every compact space is countably compact, every countably compact space is countably pracompact, and every countably pracompact space is feebly compact (see \cite{Arkhangelskii-1992} and \cite{Engelking-1989}).

A {\it topological} ({\it inverse}) {\it semigroup} is a Hausdorff topological space together with a continuous semigroup operation (and an~inversion, respectively). Obviously, the inversion defined on a topological inverse semigroup is a homeomorphism. If $S$ is a~semigroup (an~inverse semigroup) and $\tau$ is a topology on $S$ such that $(S,\tau)$ is a topological (inverse) semigroup, then we
shall call $\tau$ a (\emph{inverse}) \emph{semigroup} \emph{topology} on $S$. A {\it semitopological semigroup} is a Hausdorff topological space together with a separately continuous semigroup operation.

 The bicyclic semigroup admits only the
discrete semigroup topology and if a topological semigroup $S$ contains it as a dense subsemigroup then
${\mathscr{C}}(p,q)$ is an open subset of $S$~\cite{Eberhart-Selden-1969}. Bertman and  West in \cite{Bertman-West-1976} extended this result for the case of semitopological semigroups. Stable and $\Gamma$-compact topological semigroups do not contain the
bicyclic semigroup~\cite{Anderson-Hunter-Koch-1965, Hildebrant-Koch-1988}. The problem of an embedding of the bicyclic monoid into compact-like topological semigroups discussed in \cite{Banakh-Dimitrova-Gutik-2009, Banakh-Dimitrova-Gutik-2010, Gutik-Repovs-2007}.
In \cite{Eberhart-Selden-1969} Eberhart and Selden proved that if the bicyclic monoid ${\mathscr{C}}(p,q)$ is a dense subsemigroup of a topological monoid $S$ and $I=S\setminus{\mathscr{C}}(p,q)\neq\varnothing$ then $I$ is a two-sided ideal of the semigroup $S$. Also, there they described the closure the bicyclic monoid ${\mathscr{C}}(p,q)$ in a locally compact topological inverse semigroup. The closure of the bicyclic monoid in a countably compact (pseudocompact) topological semigroups was studied in~\cite{Banakh-Dimitrova-Gutik-2010}.

In \cite{Fihel-Gutik-2011} Fihel and Gutik showed that any Hausdorff topology $\tau$ on the extended bicyclic semigroup ${\mathscr{C}}_{\mathbb{Z}}$ such that $({\mathscr{C}}_{\mathbb{Z}},\tau)$ is a semitopological semigroup is discrete. Also in \cite{Fihel-Gutik-2011} studied a closure of the extended bicyclic semigroup ${\mathscr{C}}_{\mathbb{Z}}$ in a topological semigroup.

For any Hausdorff topology $\tau$ on an infinite semigroup of $\lambda{\times}\lambda$-matrix units $B_\lambda$ such that $(B_\lambda,\tau)$ is a semitopological semigroup every non-zero element of $B_\lambda$ is an isolated point of $(B_\lambda,\tau)$ \cite{Gutik-Pavlyk-2005a}. Also in \cite{Gutik-Pavlyk-2005a} was proved that on any infinite semigroup of $\lambda{\times}\lambda$-matrix units $B_\lambda$ there exists a unique feebly compact topology $\tau_A$ such that $(B_\lambda,\tau_A)$ is a semitopological semigroup and moroover this topology $\tau_A$ is compact. A closure of an infinite semigroup of $\lambda{\times}\lambda$-matrix units in semitopological and topological semigroups and its embeddings into compact-like semigroups were studied in \cite{Gutik-2014, Gutik-Pavlyk-2005a, Gutik-Pavlyk-Reiter-2009}.

Semigroup topologizations and closures of inverse semigroups of monotone co-finite partial bijections of some linearly ordered infinite sets, inverse semigroups of almost identity partial bijections and  inverse semigroups of partial bijections of a bounded finite rank studied in \cite{Chuchman-Gutik-2010, Chuchman-Gutik-2011, Guran-Gutik-Ravskyj-Chuchman-2015, Gutik-Lawson-Repovs-2015, Gutik-Pavlyk-Reiter-2009, Gutik-Pozdnyakova-2014, Gutik-Reiter-2009, Gutik-Repovs-2011, Gutik-Repovs-2012}.

To every directed graph $E$ one can associate a graph inverse
semigroup $G(E)$, where elements roughly correspond to possible paths in $E$.
These semigroups generalize polycyclic monoids. In \cite{Mesyan-Mitchell-Morayne-Peresse-20??} the authors investigated topologies that turn $G(E)$ into a
topological semigroup. For instance, they showed that in any such topology that is
Hausdorff, $G(E)\setminus \{0\}$ must be discrete for any directed graph $E$.
On the other hand, $G(E)$ need not be discrete in a Hausdorff semigroup
topology, and for certain graphs $E$, $G(E)$ admits a $T_1$ semigroup topology
in which $G(E)\setminus \{0\}$ is not discrete. In \cite{Mesyan-Mitchell-Morayne-Peresse-20??} the authors also described the algebraic structure and possible cardinality of the closure of
$G(E)$ in larger topological semigroups.

In this paper we show that the $\lambda$-polycyclic monoid for in infinite cardinal $\lambda\geqslant 2$ has similar algebraic properties so has the polycyclic monoid $P_n$ with finitely many $n\geqslant 2$ generators. In particular we prove that for every infinite cardinal $\lambda$ the polycyclic monoid $P_{\lambda}$ is a congruence-free, combinatorial, $0$-bisimple, $0$-$E$-unitary inverse semigroup. Also we show that every non-zero element $x$ is an isolated point in $(P_{\lambda},\tau)$ for every Hausdorff topology on $P_{\lambda}$, such that $P_{\lambda}$ is a semitopological semigroup, and every locally compact Hausdorff semigroup topology on $P_\lambda$ is discrete. The last statement extends results of the paper \cite{Mesyan-Mitchell-Morayne-Peresse-20??} obtaining for topological inverse graph semigroups. We describe all feebly compact topologies $\tau$ on $P_{\lambda}$ such that $\left(P_{\lambda},\tau\right)$ is a semitopological semigroup and its Bohr compactification as a topological semigroup. We prove that for every cardinal $\lambda\geqslant 2$ any continuous homomorphism from a topological semigroup $P_\lambda$ into an arbitrary countably compact topological semigroup is annihilating and there exists no a Hausdorff feebly compact topological semigroup which contains $P_{\lambda}$ as a dense subsemigroup.


\section{Algebraic properties of the $\lambda$-polycyclic monoid for an infinite cardinal $\lambda$}

In this section we assume that $\lambda$ is an infinite cardinal.

We repeat the thinking and arguments from \cite[Section~9.3]{Lawson-1998}.

We shall give a representation for the polycyclic monoid $P_\lambda$ by means of partial bijections on the free monoid $\mathscr{M}_\lambda$ over the cardinal $\lambda$. Put $A=\{x_i\colon i\in\lambda\}$. Then the free monoid $\mathscr{M}_\lambda$ over the cardinal $\lambda$ is isomorphic to the free monoid $\mathscr{M}_\lambda$ over the set $A$. Next we define for every $i\in\lambda$ the partial map $\alpha\colon \mathscr{M}_\lambda\to \mathscr{M}_\lambda$ by the formula $(u)\alpha_i=x_iu$ and put that $\mathscr{M}_\lambda$ is the domain and $x_i\mathscr{M}_\lambda$ is the range of $\alpha_i$. Then for every $i\in\lambda$ we may regard so defined partial map as an element of the symmetric inverse monoid $\mathscr{I}(\mathscr{M}_\lambda)$ on the set $\mathscr{M}_\lambda$. Denote by $I_\lambda$ the inverse submonoid of $\mathscr{I}(\mathscr{M}_\lambda)$ generated by the set $\{\alpha_i\colon i\in\lambda\}$. We observe that $\alpha_i\alpha_i^{-1}$ is the identity partial map on $\mathscr{M}_\lambda$ for each $i\in\lambda$ and whereas if $i\neq j$ then $\alpha_i\alpha_j^{-1}$ is the empty partial map on the set $\mathscr{M}_\lambda$, $i,j\in\lambda$. Define the map $h\colon P_\lambda\to I_\lambda$ by the formula $(p_i)h=\alpha_i$ and $(p_i^{-1})h=\alpha_i^{-1}$, $i\in\lambda$. Then by Proposition~2.3.5 of \cite{Lawson-1998}, $I_\lambda$ is a homomorphic image of $P_\lambda$ and by Proposition~9.3.1 from \cite{Lawson-1998} the map $h\colon P_\lambda\to I_\lambda$ is an isomorphism. Since the band of the semigroup $I_\lambda$ consists of partial identity maps, the identifying the semilattice of idempotents of $I_\lambda$ with the free monoid $\mathscr{M}_\lambda^0$ with adjoined zero admits the following partial order on $\mathscr{M}_\lambda^0$:
\begin{equation}\label{eq-2.1}
    u\leqslant v \quad \hbox{if and only if} \quad v \hbox{~is a prefix of~} u \quad \hbox{for~} u,v\in \mathscr{M}_\lambda^0, \qquad \hbox{and} \qquad 0\leqslant u \quad \hbox{for every~} u\in \mathscr{M}_\lambda^0.
\end{equation}
This partial order admits the following semilattice operation on $\mathscr{M}_\lambda^0$:
\begin{equation*}
    u*v=v*u=
    \left\{
      \begin{array}{ll}
        u, & \hbox{if~} v \hbox{~is a prefix of~} u;\\
        0, & \hbox{otherwise},
      \end{array}
    \right.
\end{equation*}
and $0*u=u*0=0*0=0$ for arbitrary words $u,v\in\mathscr{M}_\lambda^0$.

\begin{remark}\label{remark-2.1}
We observe that for an arbitrary non-zero cardinal $\lambda$ the set $\mathscr{M}_\lambda^0\setminus\{0\}$ with the dual partial order to (\ref{eq-2.1}) is order isomorphic to the $\lambda$-ary tree $T_\lambda$ with the countable height.
\end{remark}

Hence, we proved the following proposition.

\begin{proposition}\label{proposition-2.2}
For every infinite cardinal $\lambda$ the semigroup $P_{\lambda}$ is isomorphic to the inverse semigroup $I_{\lambda}$ and the semilattice $E(P_{\lambda})$ is isomorphic to $(\mathscr{M}_\lambda^0,*)$.
\end{proposition}

Let $n$ be any positive integer and $i_1,\ldots,i_n\in\lambda$. We put
\begin{equation*}
    P_n^\lambda\left\langle i_1,\ldots,i_n\right\rangle=\left\langle p_{i_1},\ldots,p_{i_n}, p_{i_1}^{-1},\ldots,p_{i_n}^{-1} \mid p_{i_k} p_{i_k}^{-1}=1, p_{i_k}p_{i_l}^{-1}=0 \hbox{~for~} i_k\neq i_l\right\rangle.
\end{equation*}

The statement of the following lemma is trivial.

\begin{lemma}\label{lemma-2.3}
Let $\lambda$ be an infinite cardinal and $n$ be an arbitrary positive integer. Then $P_n^\lambda\left\langle i_1,\ldots,i_n\right\rangle$ is a submonoid of the polycyclic monoid $P_{\lambda}$ such that $P_n^\lambda\left\langle i_1,\ldots,i_n\right\rangle$ is isomorphic to $P_{n}$ for arbitrary $i_1,\ldots,i_n\in\lambda$.
\end{lemma}

Our above representation of the polycyclic monoid $P_\lambda$ by means of partial bijections on the free monoid $\mathscr{M}_\lambda$ over the cardinal $\lambda$ implies the following lemma.

\begin{lemma}\label{lemma-2.4}
Let $\lambda$ be an infinite cardinal. Then for any elements $x_1,\ldots,x_k\in P_{\lambda}$ there exist $i_1,\ldots,i_n\in\lambda$ such that  $x_1,\ldots,x_k\in P_n^\lambda\left\langle i_1,\ldots,i_n\right\rangle$.
\end{lemma}

\begin{theorem}\label{theorem-2.5}
For every infinite cardinal $\lambda$ the polycyclic monoid $P_{\lambda}$ is a congruence-free combinatorial $0$-bisimple $0$-$E$-unitary inverse semigroup.
\end{theorem}

\begin{proof}
By Proposition~\ref{proposition-2.2} the semigroup $P_{\lambda}$ is inverse.

First we show that the semigroup $P_{\lambda}$ is $0$-bisimple. Then by the Munn Lemma (see \cite[Lemma~1.1]{Munn-1966} and \cite[Proposition~3.2.5]{Lawson-1998}) it is sufficient to show that for any two non-zero idempotents $e,f\in P_{\lambda}$ there exists $x\in P_{\lambda}$ such that $xx^{-1}=e$ and $x^{-1}x=f$. Fix arbitrary two non-zero idempotents $e,f\in P_{\lambda}$. By Lemma~\ref{lemma-2.4} there exist $i_1,\ldots,i_n\in\lambda$ such that  $e,f\in P_n^\lambda\left\langle i_1,\ldots,i_n\right\rangle$. Lemma~\ref{lemma-2.3}, Theorem~9.3.4 of \cite{Lawson-1998} and Proposition~3.2.5 of \cite{Lawson-1998} imply that there exists $x\in P_n^\lambda\left\langle i_1,\ldots,i_n\right\rangle\subset P_{\lambda}$ such that $xx^{-1}=e$ and $x^{-1}x=f$. Hence the semigroup $P_{\lambda}$ is $0$-bisimple.

The above representation of the polycyclic monoid $P_\lambda$ by means of partial bijections on the free monoid $\mathscr{M}_\lambda$ over the cardinal $\lambda$ implies that the $\mathscr{H}$-class in $P_{\lambda}$ which contains the unity is a singleton. Then since the polycyclic monoid $P_\lambda$ is $0$-bisimple Theorem 2.20 of \cite{Clifford-Preston-1961-1967} implies that every non-zero $\mathscr{H}$-class in $P_{\lambda}$ is a singleton. It is obvious that $\mathscr{H}$-class in $P_{\lambda}$ which contains zero is a singleton. This implies that the polycyclic monoid $P_\lambda$ is combinatorial.

Suppose to the contrary that the monoid $P_{\lambda}$ is not $0$-$E$-unitary. Then there exist a non-idempotent element $x\in P_{\lambda}$ and non-zero idempotents $e,f\in P_{\lambda}$ such that $xe=f$. By Lemma~\ref{lemma-2.4} there exist $i_1,\ldots,i_n\in\lambda$ such that  $x,e,f\in P_n^\lambda\left\langle i_1,\ldots,i_n\right\rangle$. Hence the monoid $P_n^\lambda\left\langle i_1,\ldots,i_n\right\rangle$ is not $0$-$E$-unitary, which contradicts Lemma~\ref{lemma-2.3} and Theorem~9.3.4 of \cite{Lawson-1998}. The obtained contradiction implies that the polycyclic monoid $P_{\lambda}$ is a $0$-$E$-unitary inverse semigroup.

Suppose the contrary that there exists a congruence $\mathfrak{C}$ on the polycyclic monoid $P_{\lambda}$ which is distinct from the identity and the universal congruence on $P_{\lambda}$. Then there exist distinct $x,y\in P_{\lambda}$ such that $x\mathfrak{C}y$. By Lemma~\ref{lemma-2.4} there exist $i_1,\ldots,i_n\in\lambda$ such that  $x,y\in P_n^\lambda\left\langle i_1,\ldots,i_n\right\rangle$. By Lemma~\ref{lemma-2.3} and Theorem~9.3.4 of \cite{Lawson-1998}, since the polycyclic monoid $P_n$ is congruence-free we have that the unity and zero of the polycyclic monoid $P_{\lambda}$ are $\mathfrak{C}$-equivalent and hence all elements of $P_{\lambda}$ are $\mathfrak{C}$-equivalent. This contradicts our assumption. The obtained contradiction implies that the polycyclic monoid $P_{\lambda}$ is a congruence-free semigroup.
\end{proof}

From now for an arbitrary cardinal $\lambda\geqslant 2$ we shall call the semigroup $P_\lambda$ the \emph{$\lambda$-polycyclic monoid}.

Fix an arbitrary cardinal $\lambda\geqslant 2$ and two distinct elements $a,b\in\lambda$. We consider the following subset $A=\{b^ia\colon i=0,1,2,3,\ldots\}$ of the free monoid $\mathscr{M}_\lambda$. The definition of the above defined partial order $\leqslant$ on $\mathscr{M}_\lambda^0$ implies that two arbitrary distinct elements of the set $A$ are incomparable in $(\mathscr{M}_\lambda^0,\leqslant)$. Let $B(b^ia)$ be a subsemigroup of $I_\lambda$ generated by the subset
\begin{equation*}
\left\{\alpha\in I_\lambda\colon \operatorname{dom}\alpha=b^ia\mathscr{M}_\lambda \hbox{~and~} \operatorname{ran}\alpha=b^ja\mathscr{M}_\lambda \hbox{~for some~} i,j\in\omega\right\}
\end{equation*}
of the semigroup $I_\lambda$. Since two arbitrary distinct elements of the set $A$ are incomparable in the partially ordered set $(\mathscr{M}_\lambda^0,\leqslant)$ the semigroup operation of $I_\lambda$ implies that the following conditions hold:
\begin{itemize}
  \item[$(i)$] $\alpha\beta$ is a non-zero element of the semigroup $I_\lambda$ if and only if $\operatorname{ran}\alpha=\operatorname{dom}\beta$;
  \item[$(ii)$] $\alpha\beta=0$ in $I_\lambda$ if and only if $\operatorname{ran}\alpha\neq\operatorname{dom}\beta$;
  \item[$(iii)$] if $\alpha\beta\neq 0$ in $I_\lambda$ then $\operatorname{dom}(\alpha\beta)=\operatorname{dom}\alpha$ and $\operatorname{ran}(\alpha\beta)=\operatorname{ran}\beta$;
  \item[$(iv)$] $B(b^ia)$ is an inverse subsemigroup of $I_\lambda$,
\end{itemize}
for arbitrary $\alpha,\beta\in B(b^ia)$.

Now, if we identify $\omega$ with the set of all non-negative integers $\{0,1,2,3,4,\ldots\}$, then simple verifications show that the map $\mathfrak{h}\colon B(b^ia)\rightarrow B_\omega$ defined in the following way:
\begin{itemize}
  \item[$(a)$] if $\alpha\neq 0$,  $\operatorname{dom}\alpha=b^ia\mathscr{M}_\lambda$ and $\operatorname{ran}\alpha=b^ja\mathscr{M}_\lambda$, then $(\alpha)\mathfrak{h}=(i,j)$, for $i,j\in\{0,1,2,3,4,\ldots\}$;
  \item[$(b)$] $(0)\mathfrak{h}=0$,
\end{itemize}
is a semigroup isomorphism.

Hence we proved the following proposition.

\begin{proposition}\label{proposition-2.6}
For every cardinal $\lambda\geqslant 2$ the $\lambda$-polycyclic monoid $P_{\lambda}$ contains an isomorphic copy of the semigroup of $\omega{\times}\omega$-matrix units $B_\omega$.
\end{proposition}

\begin{proposition}\label{proposition-2.8}
For every non-zero cardinal $\lambda$ and any $\alpha,\beta\in P_{\lambda}\setminus\{0\}$, both sets
 $
\left\{\chi\in P_{\lambda}\colon
\alpha\cdot\chi=\beta\right\}
 $
 and
 $
\{\chi\in P_{\lambda}\colon
\chi\cdot\alpha=\beta\}
 $
are finite.
\end{proposition}

\begin{proof}
We show that the set $\left\{\chi\in P_{\lambda}\colon \alpha\cdot\chi=\beta\right\}$ is finite. The proof in other case is similar.

It is obvious that
\begin{equation*}
\left\{\chi\in P_{\lambda}\colon \alpha\cdot\chi=\beta\right\}\subseteq \left\{\chi\in P_{\lambda}\colon \alpha^{-1}\cdot\alpha\cdot\chi=\alpha^{-1}\cdot\beta\right\}.
\end{equation*}
Then the definition of the semigroup $I_\lambda$ implies there exist words $u,v\in\mathscr{M}_\lambda$ such that the partial map $\alpha^{-1}\cdot\beta$ is the map from $u\mathscr{M}_\lambda$ onto $v\mathscr{M}_\lambda$ defined by the formula $(ux)(\alpha^{-1}\cdot\beta)=vx$ for any $x\in\mathscr{M}_\lambda$. Since $\alpha^{-1}\cdot\alpha$ is an identity partial map of $\mathscr{M}_\lambda$ we get that the partial map $\alpha^{-1}\cdot\beta$ is a restriction of the partial map $\chi$ on the set $\operatorname{dom}(\alpha^{-1}\cdot\alpha)$. Hence by the definition of the semigroup $I_\lambda$ there exists words $u_1,v_1\in\mathscr{M}_\lambda$ such that $u_1$ is a prefix of $u$, $v_1$ is a prefix of $v$ and $\chi$ is the map from $u_1\mathscr{M}_\lambda$ onto $v_1\mathscr{M}_\lambda$ defined by the formula $(u_1x)(\alpha^{-1}\cdot\beta)=v_1x$ for any $x\in\mathscr{M}_\lambda$. Now, since every word of free monoid $\mathscr{M}_\lambda$ has finitely many prefixes we conclude that the set $\left\{\chi\in P_{\lambda}\colon \alpha^{-1}\cdot\alpha\cdot\chi=\alpha^{-1}\cdot\beta\right\}$ is finite, and hence so is $\left\{\chi\in P_{\lambda}\colon \alpha\cdot\chi=\beta\right\}$.
\end{proof}

Later we need the following lemma.

\begin{lemma}\label{lemma-2.9}
Let $\lambda$ be any cardinal $\geqslant 2$. Then an element $x$ of the $\lambda$-polycyclic monoid $P_{\lambda}$ is $\mathscr{R}$-equivalent to the identity $1$ of $P_{\lambda}$ if and only if $x=p_{i_1}\ldots p_{i_n}$ for some generators $p_{i_1},\ldots,p_{i_n}\in\left\{p_i\right\}_{i\in\lambda}$.
\end{lemma}

\begin{proof}
We observe that the definition of the $\mathscr{R}$-relation implies that $x\mathscr{R}1$ if and only if $xx^{-1}=1$ (see \cite[Section~3.2]{Lawson-1998}).

$(\Rightarrow)$ Suppose that an element $x$ of $P_{\lambda}$ has a form $x=p_{i_1}\ldots p_{i_n}$. Then the definition of the $\lambda$-polycyclic monoid $P_{\lambda}$ implies that
\begin{equation*}
    xx^{-1}=\left(p_{i_1}\ldots p_{i_n}\right)\left(p_{i_1}\ldots p_{i_n}\right)^{-1}=p_{i_1}\ldots p_{i_n}p_{i_n}^{-1}\ldots p_{i_1}^{-1}=1,
\end{equation*}
and hence $x\mathscr{R}1$.

$(\Leftarrow)$ Suppose that some element $x$ of the $\lambda$-polycyclic monoid $P_{\lambda}$ is $\mathscr{R}$-equivalent to the identity $1$ of $P_{\lambda}$. Then the definition of the semigroup $P_{\lambda}$ implies that there exist finitely many $p_{i_1},\ldots,p_{i_n}\in\left\{p_i\right\}_{i\in\lambda}$ such that $x$ is an element of the submonoid $P_n^\lambda\left\langle i_1,\ldots,i_n\right\rangle$ of $P_{\lambda}$, which is generated by elements $p_{i_1},\ldots,p_{i_n}$, i.e.,
\begin{equation*}
    P_n^\lambda\left\langle i_1,\ldots,i_n\right\rangle=\left\langle p_{i_1},\ldots,p_{i_n}, p_{i_1}^{-1},\ldots,p_{i_n}^{-1}\colon p_{i_k}p_{i_k}^{-1}=1, \; p_{i_k}p_{i_l}^{-1}=0 \; \hbox{~for~} \; i_k\neq i_l \right\rangle.
\end{equation*}
Proposition~9.3.1 of \cite{Lawson-1998} implies that the element $x$ is equal to the unique string of the form $u^{-1}v$, where $u$ and $v$ are strings of the free monoid $\mathscr{M}_{\{p_{i_1},\ldots,p_{i_n}\}}$ over the set $\left\{p_{i_1},\ldots,p_{i_n}\right\}$. Next we shall show that $u$ is the empty string of $\mathscr{M}_{\{p_{i_1},\ldots,p_{i_n}\}}$. Suppose that $u=a_1\ldots a_k$ and $v=b_1\ldots b_l$, for some $a_1,\ldots,a_k,b_1,\ldots,b_l\in\left\{p_{i_1},\ldots,p_{i_n}\right\}$ and $u$ is not the empty-string of $\mathscr{M}_{\{p_{i_1},\ldots,p_{i_n}\}}$. Then the definition of the $\lambda$-polycyclic monoid $P_{\lambda}$ implies that
\begin{equation*}
\begin{split}
  xx^{-1} & =\left(u^{-1}v\right)\left(u^{-1}v\right)^{-1}=u^{-1}vv^{-1}u= \\
          & =\left(a_1\ldots a_k\right)^{-1}\left(b_1\ldots b_l\right)\left(b_1\ldots b_l\right)^{-1}\left(a_1\ldots a_k\right)= \\
          & =a_k^{-1}\ldots a_1^{-1}b_1\ldots b_lb_l^{-1}\ldots b_1^{-1}a_1\ldots a_k=\\
          & =\ldots=\\
          & =a_k^{-1}\ldots a_1^{-1}1a_1\ldots a_k=
          \\
          & =a_k^{-1}\ldots a_1^{-1}a_1\ldots a_k\neq 1,
\end{split}
\end{equation*}
which contradicts the assumption that $x\mathscr{R}1$. The obtained contradiction implies that the element $x$ has the form $x=p_{i_1}\ldots p_{i_n}$ for some generators $p_{i_1},\ldots,p_{i_n}$ from the set $\left\{p_i\right\}_{i\in\lambda}$.
\end{proof}

\section{On semigroup topologizations of the $\lambda$-polycyclic monoid}

In \cite{Eberhart-Selden-1969} Eberhart and Selden proved that if $\tau$ is a Hausdorff topology on the bicyclic monoid ${\mathscr{C}}(p,q)$ such that $({\mathscr{C}}(p,q),\tau)$ is a topological semigroup then $\tau$ is discrete. In \cite{Bertman-West-1976} Bertman and West extended this results for the case when $({\mathscr{C}}(p,q),\tau)$ is a Hausdorff semitopological semigroup. In \cite{Mesyan-Mitchell-Morayne-Peresse-20??} there proved that for any positive integer $n>1$ every non-zero element in a Hausdorff topological $n$-polycyclic monoid $P_n$ is an isolated point. The following proposition generalizes the above results.

\begin{proposition}\label{proposition-3.1}
Let $\lambda$ be any cardinal $\geqslant 2$ and $\tau$ be any Hausdorff topology on $P_{\lambda}$, such that $P_{\lambda}$ is a semitopological semigroup. Then every non-zero element $x$ is an isolated point in $(P_{\lambda},\tau)$.
\end{proposition}

\begin{proof}
We observe that the $\lambda$-polycyclic monoid $P_{\lambda}$ is a $0$-bisimple semigroup, and hence is a $0$-simple semigroup. Then the continuity of right and left translations in $(P_{\lambda},\tau)$ and Proposition~\ref{proposition-2.8} imply that it is complete to show that there exists an non-zero element $x$ of $P_{\lambda}$ such that $x$ is an isolated point in the topological space $(P_{\lambda},\tau)$.

Suppose to the contrary that the unit $1$ of the $\lambda$-polycyclic monoid $P_{\lambda}$ is a non-isolated point of the topological space $(P_{\lambda},\tau)$. Then every open neighbourhood $U(1)$ of $1$ in $(P_{\lambda},\tau)$ is infinite subset.

Fix a singleton word $x$ in the free monoid $\mathscr{M}_\lambda$. Let $\varepsilon$ be an idempotent of the $\lambda$-polycyclic monoid $P_{\lambda}$ which corresponds to the identity partial map of $x\mathscr{M}_\lambda$. Since left and right translation on the idempotent $\varepsilon$ are retractions of the topological space $(P_{\lambda},\tau)$ the Hausdorffness of $(P_{\lambda},\tau)$ implies that $\varepsilon P_{\lambda}$ and $P_{\lambda}\varepsilon$ are closed subsets of the topological space $(P_{\lambda},\tau)$, and hence so is the set $\varepsilon P_{\lambda}\cup P_{\lambda}\varepsilon$. The separate continuity of the semigroup operation and Hausdorffness of $(P_{\lambda},\tau)$ imply that for every open neighbourhood $U(\varepsilon)\not\ni 0$ of the point $\varepsilon$ in $(P_{\lambda},\tau)$ there exists an open neighbourhood $U(1)$ of the unit $1$ in $(P_{\lambda},\tau)$ such that
\begin{equation*}
U(1)\subseteq P_{\lambda}\setminus(\varepsilon P_{\lambda}\cup P_{\lambda}\varepsilon), \qquad \varepsilon\cdot U(1)\subseteq U(\varepsilon) \qquad \hbox{and} \qquad U(1)\cdot\varepsilon\subseteq U(\varepsilon).
\end{equation*}
We observe that the idempotent $\varepsilon$ is maximal in $P_{\lambda}\setminus\{1\}$. Hence any other idempotent $\iota\in P_{\lambda}\setminus(\varepsilon P_{\lambda}\cup P_{\lambda}\varepsilon)$ is incomparable with $\varepsilon$. Since the set $U(1)$ is infinite there exists an element $\alpha\in U(1)$ such that either $\alpha\cdot\alpha^{-1}$ or $\alpha^{-1}\cdot\alpha$ is an incomparable idempotent with $\varepsilon$. Then we get that either
\begin{equation*}
\varepsilon\cdot\alpha= \varepsilon\cdot(\alpha\cdot\alpha^{-1}\cdot\alpha)= (\varepsilon\cdot\alpha\cdot\alpha^{-1})\cdot\alpha= 0\cdot\alpha=0\in U(\varepsilon)
\end{equation*}
or
\begin{equation*}
\alpha\cdot\varepsilon=(\alpha\cdot\alpha^{-1}\cdot\alpha)\cdot \varepsilon= \alpha\cdot(\alpha^{-1}\cdot\alpha\cdot\varepsilon)=\alpha\cdot 0=0\in U(\varepsilon).
\end{equation*}
The obtained contradiction implies that the unit $1$ is an isolated point of the topological space $(P_{\lambda},\tau)$, which completes the proof of our proposition.
\end{proof}

A topological space $X$ is called \emph{collectionwise normal} if $X$ is $T_1$-space and for every discrete family $\left\{F_\alpha\right\}_{\alpha\in\mathscr{J}}$ of closed subsets of $X$ there exists a discrete family $\left\{S_\alpha\right\}_{\alpha\in\mathscr{J}}$ of open subsets of $X$ such that $F_\alpha\subseteq S_\alpha$ for every $\alpha\in\mathscr{J}$ \cite{Engelking-1989}.

\begin{proposition}\label{proposition-3.2}
Every Hausdorff topological space $X$ with a unique non-isoloated point is collectionwise normal.
\end{proposition}

\begin{proof}
Suppose that $a$ is a non-isolated point of $X$. Fix an arbitrary discrete family $\left\{F_\alpha\right\}_{\alpha\in\mathscr{J}}$ of closed subsets of the topological space $X$.  Then there exists an open neighbourhood $U(a)$ of the point $a$ in $X$ which intersects at most one element of the family $\left\{F_\alpha\right\}_{\alpha\in\mathscr{J}}$. In the case when $U(a)\cap F_\alpha=\varnothing$ for every $\alpha\in\mathscr{J}$ we put $S_\alpha=F_\alpha$ for all $\alpha\in\mathscr{J}$. If $U(a)\cap F_{\alpha_0}\neq\varnothing$ for some $\alpha_0\in\mathscr{J}$ we put $S_{\alpha_0}=U(a)\cup F_{\alpha_0}$ and $S_\alpha=F_\alpha$ for all $\alpha\in\mathscr{J}\setminus\{\alpha_0\}$. Then $\left\{S_\alpha\right\}_{\alpha\in\mathscr{J}}$ is a discrete family of open subsets of $X$ such that $F_\alpha\subseteq S_\alpha$ for every $\alpha\in\mathscr{J}$.
\end{proof}

Propositions~\ref{proposition-3.1} and \ref{proposition-3.2} imply the following corollary.

\begin{corollary}\label{corollary-3.2-1}
Let $\lambda$ be any cardinal $\geqslant 2$ and $\tau$ be any Hausdorff topology on $P_{\lambda}$, such that $P_{\lambda}$ is a semitopological semigroup. Then the topological space $(P_{\lambda},\tau)$ is collectionwise normal.
\end{corollary}

In \cite{Mesyan-Mitchell-Morayne-Peresse-20??} there proved that for arbitrary finite cardinal $\geqslant 2$ every Hausdorff locally compact topology $\tau$ on $P_\lambda$ such that $(P_\lambda,\tau)$ is a topological semigroup, is discrete. The following proposition extends this result for any infinite cardinal $\lambda$.

\begin{proposition}\label{proposition-3.3}
Let $\lambda$ be an infinite cardinal and $\tau$ be a locally compact Hausdorff topology on $P_\lambda$ such that $(P_\lambda,\tau)$ is a topological semigroup. Then $\tau$ is discrete.
\end{proposition}

\begin{proof}
Suppose to the contrary that there exist a Hausdorff locally compact non-discrete semigroup topology $\tau$ on $P_\lambda$. Then by Proposition~\ref{proposition-3.1} every non-zero element the semigroup $P_{\lambda}$ is an isolated point in $(P_\lambda,\tau)$. This implies that for any compact open neighbourhoods $U(0)$ and $V(0)$ of zero $0$ in $(P_\lambda,\tau)$ the set $U(0)\setminus V(0)$ is finite. Hence zero $0$ of $P_\lambda$ is an accumulation point of any infinite subset of an arbitrary open compact neighbourhood $U(0)$ of zero in $(P_\lambda,\tau)$.

Put $R_1$ is the $\mathscr{R}$-class of the semigroup $P_{\lambda}$ which contains the identity $1$ of $P_{\lambda}$. Then only one of the following conditions holds:
\begin{itemize}
  \item[$(1)$] there exists a compact open neighbourhood $U(0)$ of zero $0$ in $(P_\lambda,\tau)$ such that $U(0)\cap R_1=\varnothing$;
  \item[$(2)$] $U(0)\cap R_1$ is an infinite set for every compact open neighbourhood $U(0)$ of zero $0$ in $(P_\lambda,\tau)$.
\end{itemize}

Suppose that case $(1)$ holds. For arbitrary $x\in R_1$ we put
\begin{equation*}
    R[x]=\left\{a\in R_1\colon x^{-1}a\in U(0)\right\}.
\end{equation*}
Next we shall show that the set $R[x]$ is finite for any $x\in R_1$. Suppose to the contrary that $R[x]$ is infinite for some $x\in R_1$. Then Lemma~\ref{lemma-2.9} implies that $x^{-1}a$ is non-zero element of $P_{\lambda}$ for every $a\in R[x]$, and hence by Proposition~\ref{proposition-2.8},
\begin{equation*}
    B=\left\{x^{-1}a\colon a\in R[x]\right\}
\end{equation*}
is an infinite subset of the neighbourhood $U(0)$. Therefore, the above arguments imply that $0\in\operatorname{cl}_{P_{\lambda}}(B)$. Now, the continuity of the semigroup operation in $(P_\lambda,\tau)$ implies that
\begin{equation*}
    0=x\cdot 0\in x\cdot\operatorname{cl}_{P_{\lambda}}(B)\subseteq \operatorname{cl}_{P_{\lambda}}(x\cdot B).
\end{equation*}
Then Lemma~\ref{lemma-2.9} implies that $xx^{-1}=1$ for any $x\in R_1$ and hence we have that
\begin{equation*}
    x\cdot B=\left\{xx^{-1}a\colon a\in R[x]\right\}=\left\{a\colon a\in R[x]\right\}=R[x]\subseteq R_1.
\end{equation*}
This implies that every open neighbourhood $U(0)$ of zero $0$ in $(P_\lambda,\tau)$ contains infinitely many elements from the class $R_1$, which contradicts our assumption.

Suppose that case $(2)$ holds. Then the set $\{0\}$ is a compact minimal ideal of the topological semigroup $(P_\lambda,\tau)$. Now, by Lemma~1 of \cite{Koch-1957} (also see \cite[Vol.~1, Lemma~3,12]{Carruth-Hildebrant-Koch-1983-1986}) for every open neighbourhood $W(0)$ of zero $0$ in $(P_\lambda,\tau)$ there exists an open neighbourhood $O(0)$ of zero $0$ in $(P_\lambda,\tau)$ such that $O(0)\subseteq W(0)$ and $O(0)$ is an ideal of $\operatorname{cl}_{P_{\lambda}}(O(0))$, i.e., $O(0)\cdot \operatorname{cl}_{P_{\lambda}}(O(0))\cup\operatorname{cl}_{P_{\lambda}}(O(0))\cdot O(0)\subseteq O(0)$. But by Proposition~\ref{proposition-3.1} all non-zero elements of $P_{\lambda}$ are isolated points in $(P_\lambda,\tau)$, and hence  we have that $\operatorname{cl}_{P_{\lambda}}(O(0))=O(0)$. This implies that $O(0)$ is an open-and-closed subsemigroup of the topological semigroup $(P_\lambda,\tau)$. Therefore, the topological $\lambda$-polycyclic monoid $(P_\lambda,\tau)$ has a base $\mathscr{B}(0)$ at zero $0$ which consists of open-and-closed subsemigroups of $(P_\lambda,\tau)$. Fix an arbitrary $S\in\mathscr{B}(0)$. Then our assumption implies that there exists $x\in S\cap R_1$. Since $x\in R_1$, Lemma~\ref{lemma-2.9} implies that $xx^{-1}=1$. Without loss of generality we may assume that $x^{-1}x\neq 1$, because $S$ is a proper ideal of $P_{\lambda}$. Put $\mathbb{B}(x)=\left\langle x,x^{-1}\right\rangle$. Then Lemma~1.31 of \cite{Clifford-Preston-1961-1967} implies that $\mathbb{B}(x)$ is isomorphic to the bicyclic monoid, and since by Proposition~\ref{proposition-3.1} all non-zero elements of $P_{\lambda}$ are isolated points in $(P_\lambda,\tau)$, $\mathbb{B}^0(x)=\mathbb{B}(x)\sqcup\{0\}$ is a closed subsemigroup of the topological semigroup $(P_\lambda,\tau)$, and hence by Corollary~3.3.10 of \cite{Engelking-1989}, $\mathbb{B}^0(x)$ with the induced topology $\tau_{\mathbb{B}}$ from $(P_\lambda,\tau)$ is a Hausdorff locally compact topological semigroup. Also, the above presented arguments imply that $\left\langle x\right\rangle\cup\{0\}$ with the induced topology from $(P_\lambda,\tau)$ is a compact topological semigroup, which is contained in $\mathbb{B}^0(x)$ as a subsemigroup. But by Corollary~1 from \cite{Gutik-2015}, $(\mathbb{B}^0(x),\tau_{\mathbb{B}})$ is the discrete space, which contains a compact infinite subspace $\left\langle x\right\rangle\cup\{0\}$. Hence case $(2)$ does not hold.

The presented above arguments imply that there exists no non-discrete Hausdorff locally compact semigroup topology on the $\lambda$-polycyclic monoid  $P_\lambda$.
\end{proof}

The following example shows that the statements of Proposition~\ref{proposition-3.3} does not extend in the case when $(P_\lambda,\tau)$ is a semitopological semigroup with continuous inversion. Moreover there exists a compact Hausdorff topology $\tau_{\textsf{A-c}}$ on $P_\lambda$ such that $(P_\lambda,\tau_{\textsf{A-c}})$ is semitopological inverse semigroup with continuous inversion.

\begin{example}\label{example-3.4}
Let $\lambda$ is any cardinal $\geqslant 2$. Put $\tau_{\textsf{A-c}}$ is the topology of the one-point Alexandroff compactification of the discrete space $P_\lambda\setminus\{0\}$ with the narrow $\{0\}$, where $0$ is the zero of the $\lambda$-polycyclic monoid $P_\lambda$. Since $P_\lambda\setminus\{0\}$ is a discrete open subspace of $(P_\lambda,\tau_{\textsf{A-c}})$, it is complete to show that the semigroup operation is separately continuous in $(P_\lambda,\tau_{\textsf{A-c}})$ in the following two cases:
\begin{equation*}
    x\cdot 0 \qquad \hbox{and} \qquad 0\cdot x,
\end{equation*}
where $x$ is an arbitrary non-zero element of the semigroup $P_\lambda$. Fix an arbitrary open neighbourhood $U_A(0)$ of the zero in $(P_\lambda,\tau_{\textsf{A-c}})$ such that $A=P_\lambda\setminus U_A(0)$ is a finite subset of $P_\lambda$. By Proposition~\ref{proposition-2.8},
\begin{equation*}
    R_x^A=\left\{a\in P_\lambda\colon x\cdot a\in A \right\} \qquad \hbox{and} \qquad L_x^A=\left\{a\in P_\lambda\colon a\cdot x\in A \right\}
\end{equation*}
are finite not necessary non-empty subsets of the semigroup $P_\lambda$. Put $U_{R_x^A}(0)=P_\lambda\setminus R_x^A$, $U_{L_x^A}(0)=P_\lambda\setminus L_x^A$ and $U_{A^{-1}}=P_\lambda\setminus\{a\colon a^{-1}\in A\}$. Then we get that
\begin{equation*}
    x\cdot U_{R_x^A}(0)\subseteq U_A(0),  \qquad U_{L_x^A}(0)\cdot x\subseteq U_A(0)\qquad \hbox{and} \qquad \left(U_{A^{-1}}\right)^{-1}\subseteq U_A(0),
\end{equation*}
and hence the semigroup operation is separately continuous and the inversion is continuous in $(P_\lambda,\tau_{\textsf{A-c}})$.
\end{example}

\begin{proposition}\label{proposition-3.5}
Let $\lambda$ is any cardinal $\geqslant 2$ and $\tau$ be a Hausdorff topology on $P_\lambda$ such that $(P_\lambda,\tau)$ is a semitopological semigroup. Then the following conditions are equivalent:
\begin{itemize}
  \item[$(i)$] $\tau=\tau_{\textsf{A-c}}$;
  \item[$(ii)$] $(P_\lambda,\tau)$ is a compact semitopological semigroup;
  \item[$(iii)$] $(P_\lambda,\tau)$ is a feebly compact semitopological semigroup.
\end{itemize}
\end{proposition}

\begin{proof}
Implications $(i)\Rightarrow(ii)$ and $(ii)\Rightarrow(iii)$ are trivial and implication $(ii)\Rightarrow(i)$ follows from Pro\-position~\ref{proposition-3.1}.

$(iii)\Rightarrow(ii)$ Suppose there exists a feebly compact Hausdorff topology $\tau$ on $P_\lambda$ such that $(P_\lambda,\tau)$ is a non-compact semitopological semigroup. Then there exists an open cover $\left\{U_\alpha\right\}_{\alpha\in\mathscr{J}}$ which does not contain a finite subcover. Let $U_{\alpha_0}$ be an arbitrary element of the family $\left\{U_\alpha\right\}_{\alpha\in\mathscr{J}}$ which contains zero $0$ of the semigroup $P_\lambda$. Then $P_\lambda\setminus U_{\alpha_0}=A_{U_{\alpha_0}}$ is an infinite subset of $P_\lambda$. By Proposition~\ref{proposition-3.1}, $\left\{U_{\alpha_0}\right\}\cup\left\{\{x\}\colon x\in A_{U_{\alpha_0}}\right\}$ is an infinite locally finite family of open subset of the topological space $(P_\lambda,\tau)$, which contradicts that the space $(P_\lambda,\tau)$ is feebly compact. The obtained contradiction implies the requested implication.
\end{proof}

It is well known that the closure $\operatorname{cl}_S(T)$ of an arbitrary subsemigroup $T$ in a semitopological semigroup $S$ again is a subsemigroup of $S$ (see \cite[Proposition~I.1.8(ii)]{Ruppert-1984}). The following proposition describes the structure of a narrow of the $\lambda$-polycyclic monoid $P_\lambda$ in a semitopological semigroup.

\begin{proposition}\label{proposition-3.6}
Let $\lambda$ is any cardinal $\geqslant 2$, $S$ be a Hausdorff semitopological semigroup and $P_\lambda$ is a dense subsemigroup of $S$. Then $S\setminus P_\lambda\cup\{0\}$ is a closed ideal of $S$.
\end{proposition}

\begin{proof}
First we observe by Proposition~I.1.8(iii) from \cite{Ruppert-1984} the zero $0$ of the $\lambda$-polycyclic monoid $P_\lambda$ is a zero of the semitopological semigroup $S$. Hence the statement of the proposition is trivial when $S\setminus P_\lambda=\varnothing$.

Assume that $S\setminus P_\lambda\neq\varnothing$. Put $I=S\setminus P_\lambda\cup\{0\}$. By Theorem~3.3.9 of \cite{Engelking-1989}, $I$ is a closed subspace of $S$. Suppose to the contrary that $I$ is not an ideal of $S$. If $I\cdot S\nsubseteq I$ then there exist $x\in I\setminus\{0\}$ and $y\in P_\lambda\setminus\{0\}$ such that $x\cdot y=z\in P_\lambda\setminus\{0\}$. By Theorem~3.3.9 of \cite{Engelking-1989}, $y$ and $z$ are isolated points of the topological space $S$. Then the separate continuity of the semigroup operation in $S$ implies that there exists an open neighbourhood $U(x)$ of the point $x$ in $S$ such that $U(x)\cdot\{y\}=\{z\}$. Then we get that $|U(x)\cap P_\lambda|\geqslant\omega$ which contradicts Proposition~\ref{proposition-2.8}. The obtained contradiction implies the inclusion $I\cdot S\subseteq I$. The proof of the inclusion $S\cdot I\subseteq I$ is similar.

Now we shall show that $I\cdot I\subseteq I$. Suppose to the contrary that there exist $x,y\in I\setminus\{0\}$ such that $x\cdot y=z\in P_\lambda\setminus\{0\}$. By Theorem~3.3.9 of \cite{Engelking-1989}, $z$ is an isolated point of the topological space $S$. Then the separate continuity of the semigroup operation in $S$ implies that there exists an open neighbourhood $U(x)$ of the point $x$ in $S$ such that $U(x)\cdot\{y\}=\{z\}$. Since $|U(x)\cap P_\lambda|\geqslant\omega$ there exists $a\in P_\lambda\setminus\{0\}$ such that $a\cdot y\in a\cdot I\nsubseteq I$ which contradicts the above part of our proof. The obtained contradiction implies the statement of the proposition.
\end{proof}


\section{Embeddings of the $\lambda$-polycyclic monoid into compact-like topological semigroups}

By Theorem~5 of \cite{Gutik-Pavlyk-Reiter-2009} the semigroup of $\omega{\times}\omega$-matrix units does not embed into any countably compact topological semigroup. Then by Proposition~\ref{proposition-2.6} we have that for every cardinal $\lambda\geqslant 2$ the $\lambda$-polycyclic monoid $P_\lambda$ does not embed into any countably compact topological semigroup too.

A homomorphism $\mathfrak{h}$ from a semigroup $S$ into a semigroup $T$ is called \emph{annihilating} if there exists $c\in T$ such that $(s)\mathfrak{h}=c$ for all $s\in S$. By Theorem~6 of \cite{Gutik-Pavlyk-Reiter-2009} every continuous homomorphism from the semigroup of $\omega{\times}\omega$-matrix units into an arbitrary countably compact topological semigroup is annihilating. Then since by Theorem~\ref{theorem-2.5} the semigroup $P_\lambda$ is congruence-free Theorem~6 of \cite{Gutik-Pavlyk-Reiter-2009} and Theorem~\ref{theorem-2.5} imply the following corollary.

\begin{corollary}\label{corollary-4.1}
For every cardinal $\lambda\geqslant 2$ any continuous homomorphism from a topological semigroup $P_\lambda$ into an arbitrary countably compact topological semigroup is annihilating.
\end{corollary}

\begin{proposition}\label{proposition-4.2}
For every cardinal $\lambda\geqslant 2$ any continuous homomorphism from a topological semigroup $P_\lambda$ into a topological semigroup $S$ such that $S\times S$ is a Tychonoff pseudocompact space is annihilating, and hence $S$ does not contain the $\lambda$-polycyclic monoid $P_{\lambda}$.
\end{proposition}

\begin{proof}
First we shall show that $S$ does not contain the $\lambda$-polycyclic monoid $P_{\lambda}$. By \cite[Theorem 1.3]{Banakh-Dimitrova-2010} for any topological semigroup S with the pseudocompact square $S\times S$ the semigroup operation $\mu\colon S\times S\rightarrow S$ extends to a continuous semigroup operation $\beta\mu\colon \beta S\times \beta S\rightarrow \beta S$, so $S$ is a subsemigroup of the compact topological semigroup $\beta S$. Therefore the $\lambda$-polycyclic monoid $P_{\lambda}$ is a subsemigroup of compact topological semigroup $\beta $S which contradicts Corollary~\ref{corollary-4.1}. The first statement of the proposition implies from the statement that $P_{\lambda}$ is a congruence-free semigroup.
\end{proof}

Recall~\cite{DeLeeuw-Glicksberg-1961} that a {\it Bohr compactification of a topological semigroup $S$} is a~pair $(\beta, B(S))$ such that $B(S)$ is a compact topological semigroup, $\beta\colon S\to B(S)$ is a continuous homomorphism, and if $g\colon S\to T$ is a continuous homomorphism of $S$ into a compact semigroup $T$, then there exists a unique continuous homomorphism $f\colon B(S)\to T$ such that the diagram
\begin{equation*}
\xymatrix{
S\ar[r]^\beta\ar[d]_g & B(S)\ar[ld]^f\\ T
}
\end{equation*}
commutes.

By Theorem~\ref{theorem-2.5} for every infinite cardinal $\lambda$ the polycyclic monoid $P_{\lambda}$ is a congruence-free inverse semigroup and hence Corollary~\ref{corollary-4.1} implies the following corollary.

\begin{corollary}\label{corollary-4.3}
For every cardinal $\lambda\geqslant 2$ the Bohr compactification of a topological $\lambda$-polycyclic monoid $P_{\lambda}$ is a trivial semigroup.
\end{corollary}

The following theorem generalized Theorem~5 from \cite{Gutik-Pavlyk-Reiter-2009}.

\begin{theorem}\label{theorem-4.4}
For every infinite cardinal $\lambda$ the semigroup of $\lambda{\times}\lambda$-matrix units $B_{\lambda}$ does not densely embed into a Hausdorff feebly compact topological semigroup.
\end{theorem}

\begin{proof}
Suppose to the contrary that there exists a Hausdorff feebly compact topological semigroup $S$ which contains the semigroup of $\lambda{\times}\lambda$-matrix units $B_{\lambda}$ as a dense subsemigroup.

First we shall show that the subsemigroup of idempotents $E(B_{\lambda})$ of the semigroup $\lambda{\times}\lambda$-matrix units $B_{\lambda}$ with the induced topology from $S$ is compact. Suppose to the contrary that $E(B_{\lambda})$ is not a compact subspace of $S$. Then there exists an open neighbourhood $U(0)$ of the zero $0$ of $S$ such that $E(B_{\lambda})\setminus U(0)$ is an infinite subset of $E(B_{\lambda})$.  Since the closure of semilattice in a topological semigroup is subsemilattice (see \cite[Corollary~19]{Gutik-Pavlyk-2003}) and every maximal chain of $E(B_{\lambda})$ is finite, Theorem~9 of \cite{Stepp-1975} implies that the band $E(B_{\lambda})$ is a closed subsemigroup of $S$. Now, by Lemma~2 from \cite{Gutik-Pavlyk-2005a} every non-zero element of the semigroup $B_{\lambda}$ is an isolated point in the space $S$, and hence by Theorem~3.3.9 of \cite{Engelking-1989}, $B_{\lambda}\setminus\{0\}$ is an open discrete subspace of the topological space $S$. Therefore we get that $E(B_{\lambda})\setminus U(0)$ is an infinite open-and-closed discrete subspace of $S$. This contradicts the condition that $S$ is a feebly compact space.

If the subsemigroup of idempotents $E(B_{\lambda})$ is compact then by Theorem~1 from \cite{Gutik-Pavlyk-Reiter-2009} the semigroup of $\lambda{\times}\lambda$-matrix units $B_{\lambda}$ is closed subsemigroup of $S$ and since $B_{\lambda}$ is dense in $S$, the semigroup $B_{\lambda}$ coincides with the topological semigroup $S$. This contradicts Theorem~2 of \cite{Gutik-Pavlyk-2005a} which states that there exists no a feebly compact Hausdorff topology $\tau$ on the semigroup of $\lambda{\times}\lambda$-matrix units $B_{\lambda}$ such that $(B_{\lambda},\tau)$ is a topological semigroup. The obtained contradiction implies the statement of the theorem.
\end{proof}

\begin{lemma}\label{lemma-4.5}
Every Hausdorff feebly compact topological space with a dense discrete subspace is countably pracompact.
\end{lemma}

\begin{proof}
Suppose to the contrary that there exists a  feebly compact topological space $X$ with a dense discrete subspace $D$ such that $X$ is not countably pracompact. Then every dense subset $A$ in the topological space $X$ contains an infinite subset $B_A$ such that $B_A$ hasn't an accumulation point in $X$. Hence the dense discrete subspace $D$ of $X$ contains an infinite subset $B_D$ such that $B_D$ hasn't an accumulation point in the topological space $X$. Then $B_D$ is a closed subset of $X$. By Theorem~3.3.9 of \cite{Engelking-1989}, $D$ is an open subspace of $X$, and hence we have that $B_D$ is a closed-and-open discrete subspace of the space $X$, which contradicts the feeble compactness of the space $S$. The obtained contradiction implies the statement of the lemma.
\end{proof}

\begin{theorem}\label{theorem-4.6}
For arbitrary cardinal $\lambda\geqslant 2$ there exists no  Hausdorff feebly compact topological semigroup which contains the $\lambda$-polycyclic monoid $P_{\lambda}$ as a dense subsemigroup.
\end{theorem}

\begin{proof}
By Proposition~\ref{proposition-3.1} and Lemma~\ref{lemma-4.5} it is suffices to show that there does not exist a Hausdorff countably pracompact topological semigroup which contains the $\lambda$-polycyclic monoid $P_{\lambda}$ as a dense subsemigroup.

Suppose to the contrary that there exists a Hausdorff countably pracompact topological semigroup $S$ which contains the $\lambda$-polycyclic monoid $P_{\lambda}$ as a dense subsemigroup. Then there exists a dense subset $A$ in $S$ such that every infinite subset $B\subseteq A$ has an accumulation point in the topological space $S$. By Proposition~\ref{proposition-3.1}, $P_{\lambda}\setminus\{0\}$ is a discrete dense subspace of $S$ and hence Theorem~3.3.9 of \cite{Engelking-1989} implies that $P_{\lambda}\setminus\{0\}$ is an open subspace of $S$. Therefore we have that $P_{\lambda}\setminus\{0\}\subseteq A$. Now, by Proposition~\ref{proposition-2.6} the $\lambda$-polycyclic monoid $P_{\lambda}$ contains an isomorphic copy of the semigroup of $\omega{\times}\omega$-matrix units $B_\omega$. Then the countable pracompactness of the space $S$ implies that every infinite subset $C$ of the set $B_\omega\{0\}$ has an accumulating point in $X$, and hence the closure $\operatorname{cl}_S(B_\omega)$ is a countably pracompact subsemigroup of the topological semigroup $S$. This contradicts Theorem~\ref{theorem-4.4}. The obtained contradiction implies the statement of the theorem.
\end{proof}


\section*{Acknowledgements}

We acknowledge Alex Ravsky for his comments and suggestions.




\end{document}